\documentclass[12pt,reqno]{amsart}

\usepackage{times}
\usepackage[T1]{fontenc}
\usepackage{mathrsfs}
\usepackage{latexsym}
\usepackage[dvips]{graphics}
\usepackage{epsfig}
\usepackage{amsmath,amsfonts,amsthm,amssymb,amscd}
\input amssym.def
\input amssym.tex
\usepackage{color}
\usepackage{enumitem}
\usepackage{pgf,tikz}
\usepackage{verbatim}
\usepackage{csquotes}
\usepackage{hyperref}

\usepackage{thmtools}
\usepackage{thm-restate}

\newcommand\be{\begin{equation}}
\newcommand\ee{\end{equation}}
\newcommand\bea{\begin{eqnarray}}
\newcommand\eea{\end{eqnarray}}
\newcommand\bi{\begin{itemize}}
\newcommand\ei{\end{itemize}}
\newcommand\ben{\begin{enumerate}}
\newcommand\een{\end{enumerate}}

\newtheorem{thm}{Theorem}[section]

\newtheorem{lem}[thm]{Lemma}

\newtheorem{rek}[thm]{Remark}

\newtheorem{prob}[thm]{Problem}

\newcommand{\Z}{\ensuremath{\mathbb{Z}}}

\newcommand{\E}{\mathbb{E}}

\newcommand{\var}[1]{\operatorname{Var}\!\left( #1 \right)}
\newcommand{\cov}[1]{\operatorname{Cov}\!\left( #1 \right)}
\renewcommand{\prob}[1]{\operatorname{Prob}\left( #1 \right)}

\numberwithin{equation}{section}

\begin{document}

\title{When almost all sets are difference dominated in $\Z/n\Z$}

\author{Anand Hemmady}
\email{\textcolor{blue}{\href{mailto:ash6@williams.edu}{ash6@williams.edu}}} 
\address{Department of
  Mathematics and Statistics, Williams College, Williamstown, MA 01267}

\author{Adam Lott}
\email{\textcolor{blue}{\href{mailto:alott@u.rochester.edu}{alott@u.rochester.edu}}} 
\address{Department of Mathematics, University of Rochester, Rochester, NY 14627}

\author{Steven J. Miller}
\email{\textcolor{blue}{\href{mailto:sjm1@williams.edu}{sjm1@williams.edu}}, \textcolor{blue}{\href{mailto:Steven.Miller.Mc.96@aya.yale.edu}{Steven.Miller.Mc.96@aya.yale.edu}}}
\address{Department of
  Mathematics and Statistics, Williams College, Williamstown, MA 01267}




\date{\today}

\thanks{This work was supported by NSF Grants DMS1265673, DMS1561945, and DMS1347804, Williams College, and the Williams College Finnerty Fund.  We also thank Oleg Lazarev and Kevin O'Bryant \cite{LazarevMillerOBryant} for the TikZ code used to generate the images in this work.}

\begin{abstract}
We investigate the behavior of the sum and difference sets of $A \subseteq \Z/n\Z$ chosen independently and randomly according to a binomial parameter $p(n) = o(1)$.  We show that for rapidly decaying $p(n)$, $A$ is almost surely difference-dominated as $n \to \infty$, but for slowly decaying $p(n)$, $A$ is almost surely balanced as $n \to \infty$, with a continuous phase transition as $p(n)$ crosses a critical threshold.  

Specifically, we show that if $p(n) = o(n^{-1/2})$, then $|A-A|/|A+A|$ converges to $2$ almost surely as $n \to \infty$ and if $p(n) = c \cdot n^{-1/2}$, then $|A-A|/|A+A|$ converges to $1+\exp(-c^2/2)$ almost surely as $n \to \infty$.  In these cases, we modify the arguments of Hegarty and Miller on subsets of $\Z$ to prove our results.  When $\sqrt{ \log n} \cdot n^{-1/2} = o(p(n))$, we prove that $|A-A| = |A+A| = n$ almost surely as $n \to \infty$ if some additional restrictions are placed on $n$.  In this case, the behavior is drastically different from that of subsets of $\Z$ and new technical issues arise, so a novel approach is needed.  When $n^{-1/2} = o(p(n))$ and $p(n) = O(\sqrt{ \log n} \cdot n^{-1/2})$, the behavior of $|A+A|$ and $|A-A|$ is markedly different and suggests an avenue for further study.

These results establish a ``correspondence principle'' with the existing results of Hegarty, Miller, and Vissuet.  As $p(n)$ decays more rapidly, the behavior of subsets of $\Z/n\Z$ approaches the behavior of subsets of $\Z$ shown by Hegarty and Miller.  Moreover, as $p(n)$ decays more slowly, the behavior of subsets of $\Z/n\Z$ approaches the behavior shown by Miller and Vissuet in the case where $p(n) = 1/2$.

\end{abstract}

\maketitle

\tableofcontents

\section{Introduction}

A central object of study in additive combinatorics is the sumset of a set.  Given an abelian group $G$ (written additively) and a set $A \subseteq G$, we define its sumset $A+A := \{ a+b : a,b \in A \}$.  Similarly, we can define its difference set $A-A := \{ a-b : a,b \in A \}$.  If $|A+A| > |A-A|$, we say $A$ is \emph{sum-dominated} or a \emph{More Sums Than Differences (MSTD) set}.  If $|A-A| > |A+A|$, we say $A$ is \emph{difference-dominated}, and if $|A+A| = |A-A|$ we say $A$ is \emph{balanced}.  The most common setting for studying MSTD sets is subsets of $\Z$ (though they have been studied elsewhere as well; see, for example, \cite{MillerVissuet} and \cite{Polytopes}).  Since addition in $\Z$ is commutative but subtraction is not, we typically expect most sets to be difference-dominated.  As Nathanson \cite{Nathanson1} famously remarked,
\begin{displayquote}
``Even though there exist sets $A$ which have more sums than differences, such sets should be rare, and it must be true with the right way of counting that the vast majority of sets satisfies $|A-A| > |A+A|$.''
\end{displayquote}
Surprisingly, Martin and O'Bryant \cite{MartinOBryant} showed that a positive proportion of subsets of $\{0, \hdots, n-1 \} \subset \Z$ are sum-dominated in the limit as $n \to \infty$.  Zhao \cite{Zhao1} has shown that this proportion is around $4.5 \times 10^{-4}$.  

Martin and O'Bryant proved their result by picking sets $A \subseteq \{0, \hdots, n-1\} \subset \Z$ randomly according to a binomial parameter $p = 1/2$ (i.e., every subset is equally likely) and showing that the probability of being sum-dominated is nonzero as $n \to \infty$.  This happens because if $A$ is large enough, almost all possible sums and differences appear, so it is possible to choose $A$ carefully to be sum-dominated.  However, Hegarty and Miller \cite{HegartyMiller} showed that if $A \subseteq \{0, \hdots, n-1 \} \subset \Z$ is instead picked randomly according to a binomial parameter $p(n) = o(1)$, then the probability of being sum-dominated tends to 0 as $n \to \infty$. In some sense, this is Nathanson's ``right way of counting'' because it prevents $A$ from being too large.

In this paper, we examine subsets of $\Z/n\Z$.  Miller and Vissuet \cite{MillerVissuet} showed that if subsets of $\Z/n\Z$ are picked uniformly at random, then they are balanced with probability $1$ as $n \to \infty$.  In the style of \cite{HegartyMiller}, we instead pick subsets randomly according to a binomial parameter $p(n) = o(1)$.  Our main result is the following.

\begin{thm} \label{thm: MainResult}
Let $A \subseteq \Z/n\Z$ be a subset chosen randomly according to a binomial parameter $p(n)=o(1)$.  Let $S$, $D$ denote the random variables $|A+A|$, $|A-A|$ respectively.  We have three cases.
\begin{enumerate}
\item \label{thmpart: FastDecay} Fast decay: \\
If $p(n) \ = \ o(n^{-1/2})$, then 
\begin{enumerate}
\item 
$S \ \sim \ \frac12 (n \cdot p(n))^2$,
\item 
$D \ \sim \ (n \cdot p(n))^2$.
\end{enumerate}

\item \label{thmpart: CriticalDecay} Critical decay: \\
If $p(n) \ = \ cn^{-1/2}$, then 
\begin{enumerate}
\item $S \ \sim \ n(1 - \exp( -c^2/2))$,
\item $D \ \sim \ n(1 - \exp(1 - c^2))$.
\end{enumerate}

\item \label{thmpart: SlowDecay} Slow decay: \\
If $\sqrt{ \log n} \cdot n^{-1/2} = o(p(n))$ and $n$ is prime, then
\begin{enumerate}
\item $S \ \sim \ n$, \label{thmpart: sumslow}
\item $D \ \sim \ n$. \label{thmpart: diffslow}
\end{enumerate}
\end{enumerate}
\end{thm}

\begin{rek}
Throughout, we will point out instances where the case $n^{-1/2} = o(p(n))$ and $p(n) = O( \sqrt{\log n} \cdot n^{-1/2} )$ causes deviant behavior.
\end{rek}

\begin{rek}
In part \ref{thmpart: SlowDecay} we assume that $n$ is prime to simplify the technical details of our analysis; however, numerical simulations suggest that the behavior is the same for any $n$.  
\end{rek}

For parts (\ref{thmpart: FastDecay}) and (\ref{thmpart: CriticalDecay}) of Theorem \ref{thm: MainResult}, we modify the arguments in \cite{HegartyMiller} to work in this new environment where sums and differences are considered modulo $n$; however, for part \ref{thmpart: SlowDecay}, these methods do not work and a new approach is needed.

We first fix some notation.
\begin{itemize}
\item If $X$ is a random variable depending on $n$, we write $X \sim f(n)$ if for every $\epsilon > 0$, $\operatorname{Prob}((1-\epsilon)f(n) < X < (1+\epsilon)f(n)) \to 1$ as $n \to \infty$. 

\item If $X$ and $Y$ are two quantities depending on $n$, we also write $X \sim Y$ if $\lim_{n \to \infty} X/Y = 1$.  This abuse of notation should not cause any confusion as it will be clear from context if we are talking about a random variable or not.

\item We say $f(n) = O(g(n))$ if $\limsup_{n \to \infty} f(n)/g(n) < \infty$, and we say $f(n) = o(g(n))$ if $\lim_{n \to \infty} f(n)/g(n) = 0$.

\item To reduce clutter, we write $p$ in place of $p(n)$ and the dependence on $n$ is implied.
\end{itemize}

\section{Proof of main result: fast and critical decay cases}

To prove parts \ref{thmpart: FastDecay} and \ref{thmpart: CriticalDecay} of Theorem \ref{thm: MainResult}, we show that the expected value of each random variable is as claimed, and then show that the variable is strongly concentrated about its mean.

We use the following construction from \cite{HegartyMiller}.  Let
\begin{align}
X_k \ &= \ \# \{ \{ \{a_1,a_2 \}, \hdots, \{a_{2k-1},a_{2k}\} \}: a_i \in A,\ a_1+a_2 = \hdots = a_{2k-1}+a_{2k} \} \text{ and } \\
Y_k \ &= \ \# \{ \{ (a_1,a_2), \hdots, (a_{2k-1},a_{2k}) \} : a_i \in A,\ a_1-a_2 = \hdots = a_{2k-1}-a_{2k} \}.
\end{align}
In words, $X_k$ denotes the number of times $k$ pairs of elements from $A$ all have the same sum, and $Y_k$ denotes the number of times $k$ pairs of elements from $A$ all have the same difference. It is important to note that $X_k$ consists of \emph{unordered} pairs of elements, while $Y_k$ consists of \emph{ordered} pairs.  Since $A$ is a randomly chosen set, $X_k$ and $Y_k$ are random variables.  The idea is that $X_k$ and $Y_k$ measure the number of repeated sums and differences, so if we can control these quantities, we can control $|A+A|$ and $|A-A|$.  We have the following lemma.

\begin{lem} \label{lem: RepeatedSumsDiffs}
If $p(n) = O(n^{-1/2})$, then
\begin{enumerate}[label=(\alph*)]
\item \label{lempart: a}
$X_k \ \sim \ \frac{n^{k+1}}{k!} \left( \frac{p^2}{2} \right)^k$, \text{ and }
\item \label{lempart: b}
$Y_k \ \sim \ \frac{n^{k+1}}{k!} (p^2)^k$.
\end{enumerate}
\end{lem}

\begin{proof}
Each $k$-tuple that contributes to $X_k$ is one of two types: either all $2k$ elements are distinct, or one of the pairs is a repeated element.  Following the notation of \cite{HegartyMiller}, let $\xi_{1k}$, $\xi_{2k}$ be the number of tuples of the first type and second type, respectively.  Since every element of $A$ has $\lceil n/2 \rceil$ representations 
\hspace{-2mm} \footnote{Note that this is the fundamental difference between considering sums in the normal sense and considering sums mod $n$.  In the regular setting, the number of representations of $k$ as a sum depends on $k$, but in this setting it does not.  This difference is what causes the different constants in part (\ref{thmpart: CriticalDecay}) of Theorem \ref{thm: MainResult}.}
as the sum of two elements of $A$, we have
\begin{align} 
\xi_{1k} \ &= \ \sum_{r=0}^{n-1} \binom{\lceil n/2 \rceil}{k} \ = \ n \binom{\lceil n/2 \rceil}{k} \sim n \frac{(n/2)^k}{k!} \sim \frac{n^{k+1}}{2^k k!} \label{eq: TypeOne} \\
\xi_{2k} \ &= \  \sum_{r=0}^{n-1} \binom{\lceil n/2 \rceil}{k-1} \ = \ n \binom{\lceil n/2 \rceil}{k-1} \sim n \frac{(n/2)^{k-1}}{(k-1)!} \sim \frac{n^k}{2^{k-1} (k-1)!}. \label{eq: TypeTwo}
\end{align}
The expected value of $X_k$ is then given by
\begin{equation} \label{eq: RepeatedSums}
\E[X_k] \ = \ \xi_{1k} p^{2k} + \xi_{2k} p^{2k-1} \ = \ \frac{n^{k+1}}{2^k k!} p^{2k} + \frac{n^k}{2^{k-1} (k-1)!} p^{2k-1} \ \sim \ \frac{n^{k+1}}{k!} \left( \frac{p^2}{2} \right)^k.
\end{equation}
Now we show that the variance of $X_k$ is small enough to guarantee strong concentration about the mean.  It is sufficient to show that $\var{X_k} = o(\E[X_k]^2)$ (see, for example, chapter 4 of \cite{AlonSpencer}).  We have
\begin{equation}
\var{X_k} = \sum_{\alpha} \var{Y_{\alpha}} + \sum_{\alpha \neq \beta} \cov{Y_{\alpha},Y_{\beta}},
\end{equation}
where the sums are over $k$-tuples of unordered pairs of elements of $A$ and $Y_{\alpha}$ is an indicator variable that equals 1 if $\alpha$ contributes to $X_k$ and 0 otherwise.  From the arguments in \cite{AlonSpencer}, it is enough to show that
\begin{equation} \label{eq: JointProb}
\sum_{\alpha, \beta} \prob{\alpha, \beta \text{ both contribute to $X_k$}} = o(\E[X_k]^2),
\end{equation}
where the sum is now over all $\alpha$, $\beta$ that have at least one member in common.  The main contribution to this sum comes from pairs $\alpha$, $\beta$ with one element in common and $2k$ distinct elements each, and there are $O(n^{2k+1})$ choices for this (see the proof of Lemma 2.1 in \cite{HegartyMiller} for details).  Thus the sum (\ref{eq: JointProb}) is at most $O(n^{2k+1}p^{4k-1}) = o(n^{2k+2}p^{4k})$.  Thus part \ref{lempart: a} is proven.

The proof of part \ref{lempart: b} follows the exact same argument, so we omit the details.
\end{proof}

We can now prove parts (\ref{thmpart: FastDecay}) and (\ref{thmpart: CriticalDecay}) of Theorem \ref{thm: MainResult}.

\begin{proof}[Proof of Theorem \ref{thm: MainResult}, part (\ref{thmpart: FastDecay}).]
\text{} \\
If $p(n) = o(n^{-1/2})$, we have by Lemma \ref{lem: RepeatedSumsDiffs} that  $X_1 \sim \frac{1}{2} (n \cdot p(n))^2$, $Y_1 \sim (n \cdot p(n))^2$, $X_k = o(X_1)$, and $Y_k = o(Y_1)$ for $k \geq 2$.  In other words, all but a vanishing proportion of pairs of elements in $A$ have distinct sums and differences.  Thus $S \sim \frac{1}{2} (n \cdot p(n))^2$ and $D \sim (n \cdot p(n))^2$ as claimed.  This proves part (\ref{thmpart: FastDecay}).
\end{proof}

\begin{proof}[Proof of Theorem \ref{thm: MainResult}, part (\ref{thmpart: CriticalDecay})]
\text{} \\
By inclusion-exclusion, we have that
\begin{equation} \label{eq: InclusionExclusion}
S \ = \ \sum_{k=1}^{\infty} (-1)^{k+1} X_k.
\end{equation}
Lemma \ref{lem: RepeatedSumsDiffs} yields $X_k \sim n \frac{1}{k!} \left( \frac{c^2}{2} \right)^k$, so (\ref{eq: InclusionExclusion}) gives
\begin{equation}
S \ \sim \ n \cdot \sum_{k=1}^{\infty} \frac{(-1)^k}{k!} \left( \frac{c^2}{2} \right)^k \ = \ n(1-\exp(-c^2/2)),
\end{equation}
which was the claim.  Similarly, for differences we have
\begin{align}
D \ &= \ \sum_{k=1}^{\infty} (-1)^{k+1} Y_k \ \text{ and } \nonumber \\
Y_k \ &\sim \ n \frac{1}{k!} (c^2)^k,
\end{align}
so 
\begin{equation}
D \ \sim \ n \cdot \sum_{k=1}^{\infty} \frac{(-1)^k}{k!} (c^2)^k \ = \ n(1-\exp(-c^2)).
\end{equation}
This proves part (\ref{thmpart: CriticalDecay}).
\end{proof}

\section{Proof of main result: slow decay case}

We need the following bound.

\begin{restatable}{lem}{Lemma}
\label{lem: SeriesBound}
Suppose $p(n) = n^{-\delta}$ where $\delta \in (0,1/2)$.  Let 
\begin{equation}
F(n) \ = \ \sum_{r=0}^{n/2} \binom{n-r}{r} p^{r} (1-p)^{n-r}.
\end{equation}
Then $F(n) = o(1/n^3)$.
\end{restatable}

This is proven in Appendix \ref{app: SeriesBoundProof}.

To prove part \ref{thmpart: SlowDecay} of Theorem \ref{thm: MainResult}, we use the following strategy.  We let $S^c = n - |A+A|$ be the number of sums missing from $A+A$, and we show that 
\begin{equation}
\lim_{n \to \infty} \E[S^c] \ = \ \lim_{n \to \infty} \var{S^c} \ = \ 0.
\end{equation}
To show that this is sufficient, let $v(n) = \var{S^c}$ and let $s(n) = \sqrt{v(n)}$.  By Chebyshev's inequality
\begin{equation}
\prob{|S^c - \E[S^c]| \ \geq \ k s(n)} \ \leq \ \frac{1}{k^2}.
\label{eq: chebyshev1}
\end{equation}
Taking $k = 1/\sqrt{s(n)}$, we see that 
\begin{equation}
\prob{|S^c - \E[S^c]| \ \geq \ \sqrt{s(n)}} \leq s(n).
\label{eq: chebyshev2}
\end{equation}
Thus, since $\E[S^c]$ also tends to $0$, we can say that $\prob{S^c > 1/2} \to 0$ as $n \to \infty$; thus $S \sim n$.  We also use this argument for differences by replacing $S^c$ everywhere with $D^c := n - |A-A|$.  We can now prove part (\ref{thmpart: SlowDecay}) of Theorem \ref{thm: MainResult}.

\begin{proof}[Proof of Theorem \ref{thm: MainResult}, part (\ref{thmpart: sumslow}).]
\text{} \\
Let $S^c = n - |A+A|$.  First we compute $\E[S^c]$.  Define the random variables $Z_k$ by
\begin{equation}
Z_k \ := \ \begin{cases} 
1, & k \not\in A+A \\ 
0, & k \in A+A 
\end{cases}
\end{equation}
so that $\sum_{k \in \Z/n\Z} Z_k = S^c$. 

Since $n$ is assumed to be a large prime and is therefore odd, each $k \in \Z/n\Z$ can be written as a sum in $(n+1)/2$ different ways, and all of the representations are independent of each other, so $\prob{k \not\in A+A} = \E[Z_k] = (1-p^2)^{(n+1)/2}$.  Thus we have

\begin{equation}
\label{eq: ExpectedMissingSums}
\E[S^c] \ = \ \sum_{k \in \Z/n\Z} \E[Z_k] \ = \ n(1-p^2)^{(n+1)/2} \ \sim \ n(1-p^2)^{n/2}.
\end{equation}
Denote this quantity by $G(n)$.  To show that it tends to 0, we have
\begin{align}
\log G(n) \ &= \ \log n + \frac{1}{2}n \log (1-p^2) \nonumber \\
&= \ \log n + \frac{1}{2}n (-p^2 + O(p^4)) \nonumber \\ 
&= \ \log n - \frac{1}{2}np^2 + O(np^4),
\end{align}
which tends to $-\infty$ as $n \to \infty$ because $\log n = o(np^2)$; thus $G(n)$ tends to 0.

\begin{rek}
\label{rek: EScRemark}
If instead we had $p(n) = o(\sqrt {\log n} \cdot n^{-1/2})$, then $\log G(n)$ would tend to $+ \infty$ rather than $- \infty$.
\end{rek}

We now compute $\var{S^c}$.  We have
\begin{align}
\var{S^c} \ &= \ \sum_{k \in \Z/n\Z} \var{Z_k} + \sum_{i \neq j \in \Z/n\Z} \cov{Z_i, Z_j} \nonumber \\
\ &= \ \sum_{k} \left( \E[Z_k^2] - \E[Z_k]^2 \right) + \sum_{i \neq j} \left( \E[Z_i Z_j] - \E[Z_i] \E[Z_j] \right) \nonumber \\
\ &\sim \ \sum_{k} \left( (1-p^2)^{n/2} - (1-p^2)^n \right) \nonumber \\
\ &\quad + \ \sum_{i \neq j} \left( \prob{i \not\in A+A \ \wedge \ j \not\in A+A} - (1-p^2)^n \right) \nonumber \\
\ &\sim \ n(1-p^2)^{n/2} - n^2(1-p^2)^n + \sum_{i \neq j} \prob{i \not\in A+A \ \wedge \ j \not\in A+A}.
\end{align}

We can get an expression for the probability that $i$ and $j$ are both missing from the sumset by translating the problem into graph theory.  Define the graph $G_{n,i,j}^S$ as follows.  The vertices of $G_{n,i,j}^S$ are the elements $\{0, \hdots, n-1\}$, and vertices $a$ and $b$ are connected by an edge if and only if $a+b \equiv i \pmod{n}$ or $a+b \equiv j \pmod{n}$ (see Figure \ref{fig: SumGraph}).

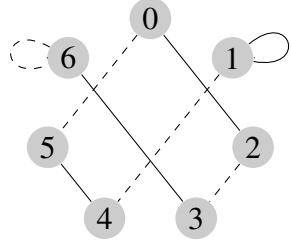
\begin{figure}[h]
\begin{tikzpicture}
  [scale=.7,auto=left,every node/.style={circle,fill=black!20,minimum size=12pt,inner sep=2pt}]
  \node (n0) at (0,2)       {0};
  \node (n1) at (1.56,1.25)   {1};
  \node (n2) at (1.95,-0.45)       {2};
  \node (n3) at (0.87,-1.8)  {3};
  \node (n4) at (-0.87,-1.8)      {4};
  \node (n5) at (-1.95,-0.45) {5};
  \node (n6) at (-1.56,1.25)      {6};

  \foreach \from/\to in {n0/n2,n1/n1,n3/n6,n4/n5}
    \draw (\from) -- (\to);
  \draw (n1) to [out=400,in=350,looseness=8] (n1);
  \foreach \from/\to in {n0/n5,n1/n4,n2/n3,n6/n6}
    \draw [dash pattern = on 3 pt off 3 pt] (\from) -- (\to);
  \draw [dash pattern = on 3 pt off 3 pt] (n6) to [out=150,in=200,looseness=8] (n6);
\end{tikzpicture}

\caption{The graph $G_{7,2,5}^S$.  For clarity, each edge is solid or dotted depending on the sum of the two incident vertices, but this doesn't affect the graph.}

\label{fig: SumGraph}
\end{figure}

The event $(i \not\in A+A \ \wedge \ j \not\in A+A)$ corresponds to the event that no two adjacent vertices of $G_{n,i,j}^S$ are in $A$.  Since we have assumed $n$ is prime, we know that for any $i,j$, $G_{n,i,j}^S$ is isomorphic to a path of $n$ vertices with a loop on each endpoint (see Figure \ref{fig: UntangledSumGraph}).  

\begin{figure}[h]

\begin{tikzpicture}
  [scale=.7,auto=left,every node/.style={circle,fill=black!20,minimum size=12pt,inner sep=2pt}]
  \node (n6) at (0,2)       {6};
  \node (n1) at (1.56,1.25)   {1};
  \node (n4) at (1.95,-0.45)       {4};
  \node (n5) at (0.87,-1.8)  {5};
  \node (n0) at (-0.87,-1.8)      {0};
  \node (n2) at (-1.95,-0.45) {2};
  \node (n3) at (-1.56,1.25)      {3};

  \foreach \from/\to in {n0/n2,n1/n1,n3/n6,n4/n5}
    \draw (\from) -- (\to);
  \draw (n1) to [out=400,in=350,looseness=8] (n1);
  \foreach \from/\to in {n0/n5,n1/n4,n2/n3,n6/n6}
    \draw [dash pattern = on 3 pt off 3 pt] (\from) -- (\to);
  \draw [dash pattern = on 3 pt off 3 pt] (n6) to [out=50,in=100,looseness=8] (n6);
\end{tikzpicture}

\caption{The graph $G_{7,2,5}^S$ from Figure \ref{fig: SumGraph} rearranged to illustrate the structure.  The graph $G_{n,i,j}^S$ has this structure for any $n,i,j$.}
\label{fig: UntangledSumGraph}
\end{figure}
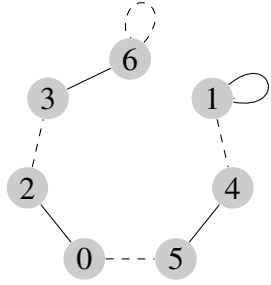

We see that $A$ can't contain either of the two endpoints (6 and 1 in the figure).  So, after a relabeling of the vertices,  picking a set $A$ so that $i$ and $j$ are both missing from $A+A$ is equivalent to picking a subset of $\{1, \hdots, n-2\}$ with no two consecutive elements (1 and $n-2$ are not considered consecutive). Since we are picking elements of $A$ independently with probability $p$, the probability of picking $A$ with no two consecutive elements is
\begin{equation}
\sum_{r=0}^{(n-2)/2} C(n-2,r) p^r (1-p)^{(n-2)-r},
\label{eq: SumVar}
\end{equation}
where $C(n-2,r)$ denotes the number of $r$-element subsets of $\{1, \hdots, n-2\}$ with no consecutive elements.  By a simple counting argument (see the calculation of quantity $Y$ in Appendix \ref{app: LucasNumbers}), we have $C(n-2,r) = \binom{n-2-r+1}{r} $.

\begin{rek}
The numbers $C(n-2,r)$ also have another combinatorial interpretation.  Any positive integer can be written uniquely as a sum of non-adjacent Fibonacci summands; these numbers are how many integers at most $F_{n-1}-1$ have exactly $r$ summands. This partition of the integers in $[0, F_{n-1}-1)$ was used in \cite{KKMW} to show that the distribution of the number of summands converges to a Gaussian as $n\to\infty$.
\end{rek}

Since the probability that neither of the endpoints gets picked is $(1-p)^2$, we have that
\begin{align}
\prob{i \not\in A+A \ \wedge \ j \not\in A+A} \ &= \ (1-p)^2 \sum_{r=0}^{(n-2)/2} \binom{n-2-r+1}{r} p^r (1-p)^{(n-2)-r} \nonumber \\
&\leq \ \sum_{r=0}^{n/2} \binom{n-r}{r} p^r (1-p)^{n-r}. \label{eq: JointSumProb}
\end{align}
Recall that (\ref{eq: JointSumProb}) is the quantity $F(n)$ from Lemma \ref{lem: SeriesBound}.  So we have
\begin{align}
\var{S^c} \ &\leq \ n(1-p^2)^{n/2} - n^2(1-p^2)^n + \sum_{i \neq j} F(n) \nonumber \\
\ &\leq \ n(1-p^2)^{n/2} - n^2(1-p^2)^n + n^2 F(n).
\end{align}

The first term is $\E[S^c]$, which tends to 0.  The second term is $\E[S^c]^2$, which also tends to 0.  The third term tends to 0 by Lemma \ref{lem: SeriesBound}, so $\var{S^c}$ tends to 0 as $n \to \infty$.  This completes the proof that $S \sim n$.
\end{proof}

\begin{proof}[Proof of Theorem \ref{thm: MainResult}, part (\ref{thmpart: diffslow}).]
\text{} \\
We let $D^c := n- |A-A|$, so $D^c$ denotes the number of differences missing from $A-A$. We will compute $\E[D^c]$ and $\var{D^c}$ and show that 
\begin{equation}
\lim_{n\to\infty} \E[D^c] \ = \ \lim_{n\to\infty} \var{D^c} \ = \ 0.
\end{equation}
Replacing all instances of $S^c$ with $D^c$ in (\ref{eq: chebyshev1}) and (\ref{eq: chebyshev2}), this implies that $D \sim n$.

To find $\E[D^c]$, we must find $P(k \notin A - A)$ for every $k \in \Z/n\Z$. First, we assume that $A \neq \emptyset$, because this happens with negligible probability since we are in the slow decay case. Because $A \neq \emptyset$, we only consider $k \neq 0$. Having fixed $k$, there are $n$ different pairs $(a,b)$ such that $a - b \equiv k \mod n$: $(k,0), (2k, k), \dots, ((n-1)k, (n-2)k), (0, (n-1)k)$.

The pairs are all ordered because subtraction isn't commutative. Then $k \notin A-A$ if and only if 
\begin{equation} 
(0 \notin A \vee k \notin A) \wedge (k \notin A \vee 2k \notin A) \wedge \dots \wedge ((n-1)k \notin A \vee 0 \notin A). 
\label{eq: DiffInNotin}
\end{equation} 
Similarly to the previous section, this lends itself to a natural graph-theoretic interpretation. We construct the graph $G_{n, k}$ with vertex set $V = \{0, 1, \dots, n-1\}$ and with edge set $E = \{\{0, k\}, \dots, \{(n-1)k, 0 \}\}$. In other words, we draw an edge between all vertices $a$ and $b$ such that $a-b \equiv k\mod n$ or $b-a \equiv k\mod n$. Then an equivalent formulation of (\ref{eq: DiffInNotin}) is that $k\notin A-A$ if and only if no two adjacent vertices of $G_{n,k}$ are in $A$.

Because we assume $n$ is prime and $k \not\equiv 0$ $($mod $n)$, all of $0, k, 2k, \dots, (n-1)k$ are distinct mod $n$, so $G_{n,k}$ is necessarily a cycle on $n$ vertices (see Figure \ref{fig: DiffGraph} for an example).

\begin{figure}[h]

\begin{tikzpicture}
  [scale=.7,auto=left,every node/.style={circle,fill=black!20,minimum size=12pt,inner sep=2pt}]
  \node (n0) at (0,2)       {0};
  \node (n2) at (1.56,1.25)   {2};
  \node (n4) at (1.95,-0.45)       {4};
  \node (n6) at (0.87,-1.8)  {6};
  \node (n1) at (-0.87,-1.8)      {1};
  \node (n3) at (-1.95,-0.45) {3};
  \node (n5) at (-1.56,1.25)      {5};

  \foreach \from/\to in {n0/n2,n2/n4,n4/n6,n6/n1, n1/n3, n3/n5, n5/n0}
    \draw (\from) -- (\to);
\end{tikzpicture}

\caption{The graph $G_{7,2}$.}
\label{fig: DiffGraph}
\end{figure}
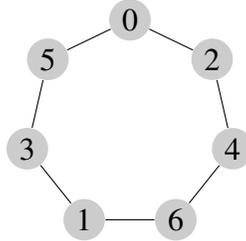

\noindent If we re-label each vertex $ak$ as $a$, then picking an $A \subseteq \Z/n\Z$ such that $k \notin A-A$ is equivalent to picking a subset of $\{0, 1, \dots, n-1\}$ such that no two consecutive elements are picked, where $0$ and $n-1$ are considered to be consecutive. By the calculation of the quantity $D(n,k)$ in Theorem \ref{thm:lucasid} from Appendix B, there are ${\binom{n-r+1}{r}} - {\binom{n-r-1}{r-2}}$ ways to choose such an $r$-element subset of $\{0,1, \dots, n-1\}$. We have then that 
\begin{equation}
P(k \notin A-A) \ = \ \sum_{r=1}^{\left \lfloor{n/2}\right\rfloor} \bigg[\binom{n-r+1}{r} - \binom{n-r-1}{r-2} \bigg] p^{r}(1-p)^{n-r}.
\label{eq:NotinDiffProb}
\end{equation}
We start the summation at $r =  1$ because we have assumed $A \neq \emptyset$. We sum until $r = \left\lfloor{n/2}\right\rfloor$ because ${\binom{n-r+1}{r}} - {\binom{n-r-1}{r-2}}$ is zero for all bigger $r$.

\begin{rek}
Here is where we rely heavily on the assumption that $n$ is prime.  If $n$ is not prime, then the graph $G_{n,k}$ becomes a union of disjoint cycles of length $n / \gcd(n,k)$, and so $\prob{k \notin A-A}$ becomes
\begin{equation}
\left( \sum_{r=1}^{\lfloor n/(2d(k)) \rfloor} \left[ \binom{n/d(k) - r + 1}{r} - \binom{n/d(k) - r - 1}{r-2} \right] p^r (1-p)^{n/d(k) - r} \right)^{d(k)},
\end{equation}
where $d(k) = \gcd(n,k)$.  Simulations suggest that as $n \to \infty$, this quantity is independent of $d(k)$, but the analysis becomes significantly more involved.
\end{rek}

We have then that
\begin{align}
\E[D^c] \ &= \ \sum_{k \in (\Z/n\Z)\setminus\{\bar{0}\}} P(x \notin A-A) \nonumber \\
&= \ (n-1)\sum_{r=1}^{\left \lfloor{n/2}\right\rfloor} \bigg[{{n-r+1} \choose {r}} - {{n-r-1} \choose {r-2}} \bigg] p^{r}(1-p)^{n-r} \nonumber \\
 &\leq \ n\sum_{r=0}^{{n/2}} \bigg[{{n-r+1} \choose {r}} - {{n-r-1} \choose {r-2}} \bigg] p^{r}(1-p)^{n-r} \nonumber \\
&= \ n\sum_{r=0}^{{n/2}} \bigg[\frac{(n-r-1)!(n^2-2nr+n)}{r!(n-2r+1)!}\bigg]p^r(1-p)^{n-r} \nonumber \\
&= \ n\sum_{r=0}^{{n/2}} \bigg[\frac{n(n-r)!}{r!(n-2r)!(n-r)}\bigg]p^r(1-p)^{n-r} \nonumber \\
&= \ n\sum_{r=0}^{{n/2}}{{n-r} \choose r} \frac{n}{n-r}p^r(1-p)^{n-r} \nonumber \\
&\leq \ 2n\sum_{r=0}^{{n/2}} {{n-r} \choose r}p^r(1-p)^{n-r} \ = \ 2nF(n), \label{eq: SimplifiedDc1}
\end{align}
and this quantity tends to 0 by Lemma \ref{lem: SeriesBound}.


We compute $\var{D^c}$ in a similar manner as $\var{S^c}$.   Define the random variables
\begin{equation}
Z'_k \ := \ \begin{cases} 
1, & k \not\in A-A \\ 
0, & k \in A-A. 
\end{cases}
\end{equation}
We have
\begin{align}
\var{D^c} \ &= \ \sum_{k \in \Z/n\Z} \var{Z'_k} + \sum_{i \neq j \in \Z/n\Z} \cov{Z'_i, Z'_j} \nonumber \\
\ &\sim \ \sum_{k \neq 0} \left( \E[(Z'_k)^2] - \E[Z'_k]^2 \right) + \sum_{i \neq j } \left( \E[Z'_i Z'_j] - \E[Z'_i] \E[Z'_j] \right) \nonumber \\
\ &\sim \ nF(n) - n^2F(n)^2 + \sum_{i \neq j} \prob{i \not\in A-A \wedge j \not\in A-A}.
\end{align}
Now note that in particular, $\prob{i \not\in A-A \wedge j \not\in A-A} \leq \prob{i \not\in A-A}$, so we have the bound
\begin{align}
\var{D^c} \ &\leq \ nF(n) - (nF(n))^2 + n(n-1)F(n) \nonumber \\
\ &\leq \ n^2F(n) - (nF(n))^2,
\end{align}
which tends to 0 by Lemma \ref{lem: SeriesBound}.  This completes the proof of part (\ref{thmpart: SlowDecay}) of Theorem \ref{thm: MainResult}.
\end{proof}


\appendix

\section{Proof of Lemma \ref{lem: SeriesBound}} 
\label{app: SeriesBoundProof}
\Lemma*
\begin{proof}
We use the following well-known approximations.

\begin{itemize}

\item Binomial approximation: if $X$ and $Y$ are two quantities depending on $n$ where $1 = o(X)$ and $Y = o(X)$, then
\begin{equation}
\binom{X}{Y} \sim \frac{X^Y}{Y!}.
\end{equation}

\item Stirling's formula: 
\begin{equation}
n! \sim \sqrt{2 \pi n} \left( \frac{n}{e} \right)^n
\end{equation}

\end{itemize}

With these at our disposal, we can prove Lemma \ref{lem: SeriesBound}.  Note that for any $r$,
\begin{equation}
\label{eq: BinomialBound}
\binom{n-r}{r} p^r (1-p)^{n-r} \ \leq \ \binom{n}{r} p^r (1-p)^{n-r}.
\end{equation}
Since the binomial distribution has mean $np$ and variance $np(1-p)$, we have by Chebyshev's inequality
\begin{equation}
\sum_{|r - np| \ \geq \ \log n \sqrt{np(1-p)}} \binom{n}{r} p^r (1-p)^{n-r} \ = \ o \left( \sum_{|r - np| \ < \ \log n \sqrt{np(1-p)}} \binom{n}{r} p^r (1-p)^{n-r} \right).
\end{equation}
What this is saying is that the tails of the distribution are negligible compared to the middle.  Thus, by (\ref{eq: BinomialBound}), we can write
\begin{equation}
\label{eq: MiddleDominates}
F(n) \ = \ \sum_{r=0}^{n/2} \binom{n-r}{r} p^r (1-p)^{n-r} \ \sim  \ \sum_{|r - np| \ < \ \log n \sqrt{np(1-p)}} \binom{n-r}{r} p^r (1-p)^{n-r}.
\end{equation}
What this means is that all but a negligible amount of the contribution to $F(n)$ comes from the terms where $r$ is close to $np$.

We have
\begin{align}
\binom{n-r}{r} p^r (1-p)^{n-r} \ &\leq \ (1-p)^n \frac{(n-r)^r}{r!} \left( \frac{p}{1-p} \right)^r \label{eq: FnBinom} \\
&\leq \ 2(1-p)^n \frac{(n-r)^r}{\sqrt{2\pi r} (r/e)^r} \left( \frac{p}{1-p} \right)^r \label{eq: FnStirling} \\
&\leq \ 2(1-p)^n \left( \frac{(n-r)ep}{r(1-p)} \right)^r. \label{eq: Gr}
\end{align}
The inequality in (\ref{eq: FnBinom}) comes from the binomial approximation and the inequality in (\ref{eq: FnStirling}) comes from Stirling's formula.  Denote the quantity on the right side of (\ref{eq: Gr}) by $g(r)$.  We now maximize $\log g(r)$ over $r \in [0, n/2]$. 
\begin{align}
\frac{g'(r)}{g(r)} \ &= \ \log \left( \frac{(n-r)pe}{r(1-p)} \right) - \frac{n}{n-r}.
\end{align}
It is clear that $g(r)$ is small at the endpoints; thus $g(r)$ is maximized at $r = r_0$ such that 
\begin{equation}
\log \left( \frac{(n-r_0)ep}{r_0(1-p)} \right) \ = \ \frac{n}{n-r_0}.
\end{equation}
We know by (\ref{eq: MiddleDominates}) that $r_0$ must satisfy $|r_0 - np| \ < \ \log n \sqrt{np(1-p)}$, so we have
\begin{equation}
\log \left( \frac{(n-r_0)ep}{r_0(1-p)} \right) \ \sim \ 1;
\end{equation}
thus $r_0 \sim np$.  Letting $r = np$ in (\ref{eq: Gr}), we now have the bound
\begin{align}
\binom{n-r}{r} p^r (1-p)^{n-r} \ &\leq \ 2(1-p)^n \left( \frac{(n-np)pe}{np(1-p)} \right)^{np} \nonumber \\
&\leq \ 2(1-p)^n e^{np} \nonumber \\ 
&= \ 2(e^p - pe^p)^n,
\end{align}
so that
\begin{equation}
F(n) \ \leq \ 2n(e^p - pe^p)^n.
\end{equation}
To complete the proof, it suffices to show that $h(n) := 2n^4(e^p - pe^p)^n \to 0$ as $n \to \infty$.  We have
\begin{align}
\log h(n) \ &= \ \log 2 + 4 \log n + n \log (e^p(1-p)) \nonumber \\
&= \ \log n + n \log(1-p) + np \nonumber \\
&= \ \log n + n( -p - \frac{1}{2} p^2 + O(p^3) ) + np \nonumber \\
&= \ \log n - \frac{1}{2} np^2 + O(np^3)
\end{align}
and this tends to $-\infty$ as $n \to \infty$ because $\log n = o(np^2)$; thus $h(n)$ tends to 0.  This completes the proof of Lemma \ref{lem: SeriesBound}.
\end{proof}

\begin{rek}
As in Remark \ref{rek: EScRemark}, if $p(n) = o(\sqrt{ \log n} \cdot n^{-1/2} )$, then $\log h(n)$ tends to $+\infty$ rather than $-\infty$.
\end{rek}

\section{Note on Lucas numbers}
\label{app: LucasNumbers}

The Lucas numbers are defined by the recurrence $$L_n = L_{n-1} + L_{n-2}$$ with initial values $L_0 = 2$ and $L_1 = 1$. Combinatorially, the $n$-th Lucas number represents the number of subsets of $\{1, \dots, n\}$ containing no consecutive integers, where $1$ and $n$ are counted as consecutive (see \cite{honsberger} for a proof of this).  An equivalent formulation of the following formula appears on page 173 of \cite{Koshy}, but we use a different counting argument to establish it directly.  We prove the following:

\begin{thm} \label{thm:lucasid}
For all $n \geq 2$, 
$$
L_n \ = \ \sum_{k=0}^{\lfloor n/2 \rfloor} \left[ \binom{n-k+1}{k} - \binom{n-k-1}{k-2} \right].
$$
\end{thm}

\begin{proof}
Let $D(n,k)$ denote the number of $k$-element subsets of $\{1, \hdots, n\}$ containing no two consecutive integers, where $1$ and $n$ are considered consecutive.  Note that for any $k > n/2$, the pigeonhole principle forces $D(n,k)=0$.  Thus
\begin{equation} \label{Dnk}
\sum_{k=0}^{\lfloor n/2 \rfloor} D(n,k) \ = \ L_n,
\end{equation}
and we just need to show $D(n,k) = \binom{n-k+1}{k} - \binom{n-k-1}{k-2}$.  For fixed $n,k$, let 
\begin{align*}
Y &= \text{\# acceptable subsets without considering $1,n$ consecutive} \\
Z &= \text{\# subsets that contain both $1$ and $n$ but no other consecutive integers} 
\end{align*}
and note that $D(n,k) = Y-Z$. Note also that $Y = C(n,k)$ from (\ref{eq: SumVar}).

To count $Y$, we use a standard stars-and-bars argument.  Suppose you have $n$ objects in a row, and you need to select $k$ of them, no two of which are consecutive.  Remove $k$ of the objects.  You now need to reinsert the $k$ objects into the row such that no two are consecutive, which means you have $n-k+1$ spots to choose from (one spot in between each remaining pair of objects and one on each end of the row).  Thus the number of ways to pick $k$ non-consecutive elements from a row of $n$ is $\binom{n-k+1}{k}$.  

Now note that to count $Z$, we just repeat the argument for $Y$, but this time we are picking $k-2$ non-consecutive elements from $\{3, \hdots, n-2\}$, and there are $\binom{(n-4)-(k-2)+1}{k-2} = \binom{n-k-1}{k-2}$.  So $D(n,k) = \binom{n-k+1}{k} - \binom{n-k-1}{k-2}$.  
\end{proof}



\vspace{1cm}

\end{document}